\documentclass[12pt,english]{article}
\usepackage[utf8]{inputenc}
\usepackage[T1]{fontenc}
\usepackage{lmodern}
\usepackage[a4paper]{geometry}

\usepackage{amsmath}
\usepackage{amsfonts}
\usepackage{amssymb}
\usepackage{amsthm}
\usepackage{bbm}
\usepackage{graphicx}
\usepackage{color}

\usepackage{babel}

\newtheorem{thm}{Théorème}
\newtheorem{prop}[thm]{Proposition}

\newtheorem{cor}[thm]{Corollaire}

\newtheorem{thm_e}[thm]{Theorem}

\newtheorem{cor_e}[thm]{Corollary}

\renewcommand*{\P}{\mathbbm{P}}
\newcommand*{\esp}{\mathbbm{E}}
\DeclareMathOperator{\Var}{Var}

\DeclareMathOperator{\Tr}{Tr}
\newcommand*{\Ncal}{\mathcal{N}} 
\newcommand*{\indic}{\mathbbm{1}}

\newcommand*{\ensembledenombres}{\mathbb}
\newcommand*{\R}{\ensembledenombres{R}}

\title{Eigenvalue variance bounds for covariance matrices}
\author{S. Dallaporta}
\date {University of Toulouse and CMLA, ENS Cachan}
\begin{document}
\maketitle
 \bigskip 

{{\bf Abstract.} 
{\it This work is concerned with finite range bounds on the variance of individual eigenvalues
of random covariance matrices, both in the bulk and at the edge of the spectrum. In a preceding paper, the author established analogous results for Wigner matrices \cite{Da_2012_variance_bound_wigner} and stated the results for covariance matrices. They are proved in the present paper. Relying on the LUE example, which needs to be investigated first, the main bounds are extended to complex covariance matrices by means of the Tao, Vu and Wang Four Moment Theorem and recent localization results by Pillai and Yin. The case of real covariance matrices is obtained from interlacing formulas.}

\bigskip \bigskip \bigskip


Random covariance matrices, or Wishart matrices, were introduced by the statistician Wishart in $1928$ to model tables of random data in multivariate statistics. The spectral properties of these
matrices are indeed useful for example for studying the properties of certain random vectors, elaborating statistical tests and for principal component analysis. Similarly to
Wigner matrices, which were introduced by the physicist Wigner in the fifties in order to study
infinite-dimensional operators in statistical physics, the asymptotic spectral properties were soon
conjectured to be universal in the sense they do not depend on the distribution of the entries (see for example \cite{AnGuZe_2010_book} and \cite{Me_1991_book}). Eigenvalues
were studied asymptotically both at the global and local regimes, considering for instance the global
behavior of the spectrum, the behavior of extreme eigenvalues or the spacings between eigenvalues in the bulk
of the spectrum. In the Gaussian case, the eigenvalue joint distribution is explicitly known, allowing for a
complete study of the asymptotic spectral properties (see for example \cite{AnGuZe_2010_book},
\cite{BaSi_2010_book}, \cite{PaSh_2011_book}). One of the main goals of random matrix theory over the past
decades was to extend these results to non-Gaussian covariance matrices.

However, in multivariate statistics, quantitative finite-range results are more useful than asymptotic properties. Furthermore, random covariance matrices have become useful in several other fields, such as compressed sensing (see
\cite{Ve_2012_non_asymptotic}), wireless communication and quantitative finance (see \cite{BaSi_2010_book}). In these fields too, quantitative results are of high interest. Several recent developments have thus been concerned with non-asymptotic
random matrix theory. See for example some recent surveys and papers on this topic \cite{RuVe_2010_non_asymptotic_extreme_singular_values}, \cite{Ve_2012_non_asymptotic} and \cite{Ve_2012_proximite_matrice_covariance_empirique}. In this paper, we investigate in this respect variance bounds on the eigenvalues of families of covariance matrices. In a preceding paper \cite{Da_2012_variance_bound_wigner}, we established similar bounds for Wigner matrices and the results for covariance matrices were stated but not proved. In the present paper, we provide the corresponding proofs. For the sake of completeness and in order to make the present paper readable separately, we reproduce here some parts of the previous one \cite{Da_2012_variance_bound_wigner}.


Random covariance matrices are defined by the following. Let $X$ be a $m\times n$ (real or complex) matrix,
with $m
\geqslant n$, such that its entries are independent, centered and have variance $1$. Then
$S_{m,n}=\frac{1}{m}X^*X$ is a covariance matrix. An important example is the case when the entries of $X$ are
Gaussian. Then $S_{m,n}$ belongs to the so-called Laguerre Unitary Ensemble (LUE) if the entries of $X$ are
complex and to the Laguerre Orthogonal Ensemble (LOE) if they are real. $S_{m,n}$ is Hermitian (or real
symmetric) and therefore has $n$ real eigenvalues. As $m \geqslant n$, none of these eigenvalues is trivial.
Furthermore, these eigenvalues are nonnegative and will be denoted by $0 \leqslant \lambda_1 \leqslant \dots
\leqslant \lambda_n$.

Among universality results, the classical Marchenko-Pastur theorem states that, if $\frac{m}{n} \to \rho
\geqslant 1$ when $n$ goes to infinity, the empirical spectral measure
$L_{m,n}=\frac{1}{n}\sum_{j=1}^n\delta_{\lambda_j}$ converges almost surely to a deterministic measure,
called the Marchenko-Pastur distribution of parameter $\rho$. This measure is compactly supported and is
absolutely continuous with respect to Lebesgue measure, with density
\[ d\mu_{MP(\rho)}(x)=\frac{1}{2\pi
x}\sqrt{(b_{\rho}-x)(x-a_{\rho})}\mathbbm{1}_{[a_{\rho},b_{\rho}]}(x)dx, \]
where $a_{\rho}=(1-\sqrt{\rho})^2$ and $b_{\rho}=(1+\sqrt{\rho})^2$ (see for example \cite{BaSi_2010_book}).
We denote by $\mu_{m,n}$ the approximate Marchenko-Pastur density
\[\mu_{m,n}(x) =
\frac{1}{2\pi x}\sqrt{(x-a_{m,n})(b_{m,n}-x)}\mathbbm{1}_{[a_{m,n},b_{m,n}]}(x), \]
with $a_{m,n}=\big(1-\sqrt{\frac{m}{n}}\big)^2$ and $b_{m,n}=\big(1+\sqrt{\frac{m}{n}}\big)^2$. The behavior of individual eigenvalues was more difficult to achieve. At the edge of the spectrum, it was proved by Bai \emph{et al} (see \cite{BaKrYi_1988_largest_covariance}, \cite{BaSiYi_1988_largest_covariance} and \cite{BaYi_1993_smallest_covariance}) under a condition on the fourth moments of the entries that, almost surely,
\begin{equation}\label{eq_lgn_eigenvalue_covariance}
a_{m,n}-\lambda_1 \underset{n \to \infty}{\to} 0 \quad \textrm{and} \quad  \lambda_n-b_{m,n} \underset{n \to \infty}{\to} 0.
\end{equation}
Once the behavior of eigenvalues at the edge of the spectrum is known, some local information on eigenvalues in the bulk can be deduced from the Marchenko-Pastur theorem. Indeed the Glivenko-Cantelli Theorem gives that almost surely
\[ \sup_{x \in \R} |F_{m,n}(x)-G_{m,n}(x)| \underset{n \to \infty}{\to} 0, \]
where $F_{m,n}$ is the distribution function of the empirical spectral distribution $L_{m,n}$ and $G_{m,n}$ is the distribution function of the approximate Marchenko-Pastur law. Combining this with crude bounds on the Marchenko-Pastur density and with \eqref{eq_lgn_eigenvalue_covariance} leads to the following law of large numbers. For all $\eta>0$, for all $\eta n \leqslant j\leqslant (1-\eta)n$, i.e. for eigenvalues in the bulk of the spectrum, 
\[ \lambda_j- \gamma_j^{m,n} \underset{n \to \infty}{\to} 0,\]
almost surely, where the theoretical location $\gamma_j^{m,n} \in [a_{m,n},b_{m,n}]$ of the $j$-th eigenvalue  $\lambda_j$ is defined by
\[\frac{j}{n}=\int_{a_{m,n}}^{\gamma_j^{m,n}} \mu_{m,n}(x)\,dx. \]
At the fluctuation level, the behavior of individual eigenvalues depends heavily on their location in the spectrum and on the value of the parameter $\rho$, at least for the smallest eigenvalues.
Indeed, when $\rho>1$, the left-side of the limiting support $a_{\rho}$ is positive. As a consequence, eigenvalues, and in particular smallest eigenvalues, can be
less than $a_{\rho}$, which is therefore called a soft edge. On the contrary, when $\rho=1$, $a_{\rho}=0$ and
no eigenvalue can be less than $a_{\rho}$. In this case, the left-side is called a hard edge. Even the
behavior of the Marchenko-Pastur density is different at the lower edge in these two cases. Indeed, when $\rho>1$, the Marchenko-Pastur density function is bounded whereas it goes to $\infty$ when $x \to 0$ if $\rho=1$. Therefore, the
behavior of the smallest eigenvalue is expected to be different according to $\rho$. Indeed, on the one hand, when $m=n$ (which implies $\rho=1$), Edelman proved that, for LUE matrices,
\[n^2\lambda_1 \overset{(d)}{\underset{n \to \infty}{\to}} \mathcal{E}(1), \]
where $\mathcal{E}(1)$ is an exponential random variable with parameter $1$. A similar result is available for
LOE matrices, see \cite{Ed_1988_condition_number} for more details. This theorem was later extended to more
general covariance matrices by Tao and Vu in \cite{TaVu_2010_smallest_singular_value}. On the other hand, when
$\rho>1$, Borodin and Forrester proved that, for LUE matrices,
\[n^{2/3}\frac{a_{m,n}-\lambda_1}{a_{m,n}^{2/3}
\big(\frac{m}{n}\big)^{-1/6}} \overset{(d)}{\underset{n \to \infty}{\to}}
F_2, \]
where $F_2$ is the so-called Tracy-Widom law (see \cite{BoFo_2003_hard_soft_edge_transition}). A similar
result holds for LOE matrices. These theorems were later extended to some non-Gaussian covariance matrices by
Feldheim and Sodin in \cite{FeSo_2010_smallest_eigenvalue_covariance} and then to large families of covariance
matrices by Wang (see \cite{Wa_2011_covariance_edge}). On the contrary, the
behavior of the largest eigenvalue relies much less on the value of the
parameter $\rho$. Indeed Johansson (see
\cite{Jo_2000_largest_eigenvalue_complex}) proved that, for LUE matrices,
\[n^{2/3}\frac{\lambda_n-b_{m,n}}{b_{m,n}^{2/3}
\big(\frac{m}{n}\big)^{-1/6}} \overset{(d)}{\underset{n \to \infty}{\to}}
F_2. \]
Johnstone proved a similar result for LOE matrices (see \cite{Jo_2001_largest_eigenvalue_real}).
Soshnikov and Péché extended these theorems to more general covariance matrices in
\cite{So_2002_largest_covariance} and \cite{Pe_2009_largest_covariance}.
They were then extended to large families of non-Gaussian covariance matrices by
Wang in \cite{Wa_2011_covariance_edge}. From these central limit theorems, the variances of the smallest (when
$\rho>1$) and largest eigenvalues are guessed to be of the order of $n^{-4/3}$.

In the bulk of the spectrum, i.e. for all eigenvalues $\lambda_j$ such that $\eta n \leqslant j \leqslant
(1-\eta)n$ for a fixed $\eta>0$, Su proved in \cite{Su_2006_fluctuations} that
\[\mu_{m,n}(\gamma_j^{m,n})\frac{\lambda_j-\gamma_j^{m,n}}{\sqrt{\frac{1}{2\pi^2}\frac{\log n}{n^2}}} \overset{(d)}{\underset{n \to \infty}{\to}}
\mathcal{N}(0,1), \]
in distribution. As for the largest eigenvalue, the value of parameter $\rho$ does not change significantly the behavior of eigenvalues in the bulk. This Central Limit Theorem was extended to families of non-Gaussian
matrices by Tao and Vu
in \cite{TaVu_2012_covariance}. The variances of eigenvalues in the bulk are then guessed to be of the order
of
$\frac{\log n}{n^2}$. Su proved in \cite{Su_2006_fluctuations} a similar Central Limit Theorem for right-side intermediate eigenvalues, which means eigenvalues $\lambda_j$ with $\frac{j}{n} \to 1$ and $n-j \to \infty$ when $n$ goes to infinity. From this theorem, the variance of such eigenvalues is guessed to be of the order of $\frac{\log (n-j)}{n^{4/3}(n-j)^{2/3}}$. This theorem was later extended to non-Gaussian covariance matrices by Wang in \cite{Wa_2011_covariance_edge}. It seems that a similar result holds for left-side intermediate eigenvalues when $\rho>1$ but Su did not carry out the computations in this case.

The aim of this paper is to provide sharp non-asymptotic bounds for the variance of individual eigenvalues of
covariance matrices. For simplicity, we basically assume that $\rho>1$. More precisely, we assume that $1<A_1 \leqslant \frac{m}{n} \leqslant A_2$ (where $A_1$ and $A_2$ are fixed constants). When $m=n$ (therefore $\rho=1$), it is possible to show that the following results in the bulk and on the right-side of the spectrum are true. It will be specified in the corresponding sections. Assume furthermore that $S_{m,n}$ is a complex covariance matrix (respectively real) whose entries have an exponential decay and have the same first four moments as those of a LUE (respectively LOE) matrix. This condition is called condition $(C0)$ and will be detailed in Section \ref{section_covariance}. The main results of this paper are the following theorems.

\begin{thm_e}[in the bulk of the spectrum]\label{thm_bulk_covariance}
For all $0 <\eta \leqslant \frac{1}{2}$, there exists a constant $C>0$ depending only on $\eta$, $A_1$ and $A_2$
such that, for all $\eta n \leqslant j \leqslant (1-\eta)n$,
\begin{equation}\label{eq_bound_covariance_bulk}
\Var(\lambda_j) \leqslant C\frac{\log n}{n^2}.
\end{equation}
\end{thm_e}

\begin{thm_e}[between the bulk and the edge of the spectrum]\label{thm_intermediate_covariance}
There exists a constant $\kappa>0$ (depending on $A_1$ and $A_2$) such that the following holds. For all $K>\kappa$, for all $0<\eta \leqslant \frac{1}{2}$, there exists a constant $C>0$~(depending on $K$, $\eta$, $A_1$ and $A_2$) such that for all covariance matrix $S_{m,n}$, for all $(1-\eta)n \leqslant j \leqslant n-K\log n$,
\begin{equation}\label{eq_bound_covariance_intermediate}
 \Var(\lambda_j) \leqslant C\frac{\log (n-j)}{n^{4/3}(n-j)^{2/3}}.
\end{equation}
\end{thm_e}

\begin{thm_e}[at the edge of the spectrum]\label{thm_edge_covariance}
There exists a constant $C>0$ depending only on $A_1$ and $A_2$ such that,
\begin{equation}\label{eq_bound_covariance_edge}
\Var(\lambda_n) \leqslant Cn^{-4/3}.
\end{equation}
\end{thm_e}
It should be mentioned that Theorem \ref{thm_intermediate_covariance} (respectively Theorem \ref{thm_edge_covariance}) probably holds for left-hand side intermediate eigenvalues (respectively the smallest eigenvalue $\lambda_1$), when $\rho>1$. We refer to Section \ref{section_intermediate_LUE} for more details on that topic. On the contrary, when $\rho=1$, the behavior of eigenvalues on the left-side of the spectrum is probably very different and much more difficult to study.

The first two theorems do not seem to be known even for LUE matrices. The first step is then to establish these
results for such matrices. The proof relies on the fact that the eigenvalues of a LUE matrix form a
determinantal process and therefore that the eigenvalue counting function has the same distribution as a sum
of independent Bernoulli variables \cite{BeKrPeVi_2006_determinantal}. Using a standard concentration inequality for Bernoulli variables, it
is then possible to establish a deviation inequality for individual eigenvalues. A simple integration leads
to the desired bounds on the variances. On the contrary, Theorem \ref{thm_edge_covariance} on the largest eigenvalue $\lambda_n$ of LUE matrices has been known for some time, at least for the largest eigenvalue $\lambda_n$ (see
\cite{LeRi_2010_deviations}). From these results for the LUE, Theorems \ref{thm_bulk_covariance}, \ref{thm_intermediate_covariance} and
\ref{thm_edge_covariance} are
then extended to large families of non-Gaussian covariance matrices by means of localization properties by
Pillai and Yin (see \cite{PiYi_2011_covariance}) and the Four Moment Theorem by Tao-Vu and Wang (see \cite{TaVu_2012_covariance} and \cite{Wa_2011_covariance_edge}).
While the localization properties almost yield the correct order on the variance, the Four Moment Theorem is
used to reach the optimal bound via a comparison with LUE matrices. Theorems \ref{thm_bulk_covariance}, \ref{thm_intermediate_covariance} and
\ref{thm_edge_covariance} are established first in the complex case. The real case is then achieved by means
of
interlacing formulas. Note that similar inequalities hold for higher moments of the eigenvalues. The proofs are exactly the same.

As a corollary of the preceding results on the variances and provided Theorem \ref{thm_intermediate_covariance} holds also for left-hand side intermediate eigenvalues, a bound on the rate of convergence of the empirical spectral distribution $L_{m,n}$ towards the Marchenko-Pastur distribution can be achieved. It can be written in terms of the $2$-Wasserstein distance between the approximate Marchenko-Pastur distribution $\mu_{m,n}$ and $L_{m,n}$, defined by the following. For $\mu$ and $\nu$ two probability measures on $\R$,
\[W_2(\mu,\nu)=\inf \bigg (\int_{\mathbb{R}^2}|x-y|^2\,d\pi(x,y)\bigg)^{1/2},\]
 where the infimum is taken over all probability measure $\pi$ on $\mathbb{R}^2$ such that its first marginal is $\mu$ and its second marginal is $\nu$. Note that the rate of convergence of this empirical distribution has also been investigated in terms of the Kolmogorov distance between $L_{m,n}$ and $\mu_{m,n}$ (see for example \cite{GoTi_2004_convergence_covariance} and \cite{GoTi_2011_convergence_covariance}). This distance is defined by
\[d_K(L_{m,n},\mu_{m,n})= \sup_{x \in \mathbb{R}}
\big|\tfrac{1}{n}\mathcal{N}_x-G_{m,n}(x)\big|,\] 
where $\mathcal{N}_x$ is the eigenvalue counting function, i.e. $\mathcal{N}_x$ is the number of eigenvalues in $(-\infty,x]$, and $G_{m,n}$ is the distribution function of the approximate Marchenko-Pastur distribution. Götze and Tikhomirov recently showed that, with high probability,
\[d_K(L_{m,n},\mu_{m,n}) \leqslant \frac{(\log n)^c}{n}\]
for some universal constant $c>0$ (see \cite{GoTi_2011_convergence_covariance}). The rate of convergence in terms of the $1$-Wasserstein distance $W_1$, also called the Kantorovich-Rubinstein distance, was studied by Guionnet and Zeitouni in \cite{GuZe_2000_concentration_spectral_measure}, who proved that $\esp[W_1(L_{m,n},\esp[L_{m,n}])]$ is bounded by $Cn^{-2/5}$. The following statement is concerned with the expectation of $W_2(L_{m,n},\mu_{m,n})$.
\begin{cor_e}\label{cor_wasserstein_2_covariance} Let $1<A_1<A_2$. Then there exists a constant $C>0$ depending only on $A_1$ and $A_2$ such that, for all $m$ and $n$ such that $1<A_1 \leqslant \frac{m}{n} \leqslant A_2$,
\begin{equation}
\esp\big[W_2^2(L_{m,n},\mu_{m,n})\big]\leqslant C\frac{\log n}{n^2}.
\end{equation}
\end{cor_e}
The proof of this corollary relies on the fact that $\esp[W_2^2(L_{m,n},\mu_{m,n})]$
is bounded above, up to a constant, by the sum of the expectations
$\esp[(\lambda_j-\gamma_j^{m,n})^2]$. The previously established bounds
then easily yield the result, provided Theorem \ref{thm_intermediate_covariance} holds for left-hand side intermediate eigenvalues.

Turning now to the content of this paper, Section \ref{LUE} describes Theorems \ref{thm_bulk_covariance}, \ref{thm_intermediate_covariance} and
\ref{thm_edge_covariance} in the LUE case. Section \ref{section_covariance} starts with the Localization Theorem of
Pillai and Yin
(see \cite{PiYi_2011_covariance}) and the Four Moment Theorem of Tao-Vu and Wang (see \cite{TaVu_2012_covariance} and \cite{Wa_2011_covariance_edge}). Theorems
\ref{thm_bulk_covariance}, \ref{thm_intermediate_covariance} and \ref{thm_edge_covariance} are then established for families of covariance
matrices. Section \ref{real_covariance}
is devoted to real matrices. Section~\ref{section_wasserstein_covariance} deals with Corollary \ref {cor_wasserstein_2_covariance} and the rate of convergence of $L_{m,n}$
towards $\mu_{MP(\rho)}$ in terms of $2$-Wasserstein distance.

Throughout this paper, $C$ and $c$ will denote positive constants, which depend on the indicated parameters and whose values may change from one line to another.

\section[Variance bounds for LUE matrices]{Deviation inequalities and variance bounds for LUE matrices}\label{LUE}
This section is concerned with Gaussian covariance matrices. The results and techniques used here heavily rely on the Gaussian structure, in particular on the determinantal properties of the eigenvalues. As a consequence of this determinantal structure, the eigenvalue counting function is known to have the same distribution as a sum of independent Bernoulli variables (see \cite{BeKrPeVi_2006_determinantal}, \cite{AnGuZe_2010_book}). Its mean and variance were computed by Su (see \cite{Su_2006_fluctuations}). Deviation inequalities can therefore be established for individual eigenvalues, leading to the announced bounds on the variance. All the proofs are written in the case when $1<A_1 \leqslant \frac{m}{n} \leqslant A_2$. Assuming $m=n$, if the results still hold, the proofs are very similar and are therefore not reproduced. 

\subsection{Inside the bulk of the spectrum}\label{deviation_LUE_bulk}
The aim of this section is to prove the following theorem for eigenvalues in the bulk, i.e. for $\lambda_j$ with $\eta n \leqslant j \leqslant (1-\eta)n$.
\begin{thm_e}\label{thm_LUE_bulk}
Let $1<A_1<A_2$. Let $S_{m,n}$ be a LUE matrix. For any $0<\eta \leqslant \frac{1}{2}$, there exists a constant $C>0$ depending only on $\eta$, $A_1$ and $A_2$ such that for all $A_1 \leqslant \frac{m}{n} \leqslant A_2$ and all $\eta n
\leqslant j \leqslant (1-\eta)n$,
\begin{equation}
\esp\big[|\lambda_j-\gamma_j^{m,n}|^2\big] \leqslant C\frac{\log n}{n^2}.
\end{equation}
In particular,
\begin{equation}
\Var(\lambda_j) \leqslant C\frac{\log n}{n^2}.
\end{equation}
\end{thm_e}
The proof of this theorem relies on the properties of the eigenvalue counting function, denoted by $\Ncal_t=\sum_{i=1}^n\indic_{\lambda_i \leqslant t}$ for every $t \in \mathbb{R}$. As announced, $\Ncal_t$ has the same distribution as a sum of independent Bernoulli variables \cite{BeKrPeVi_2006_determinantal}. Consequently, sharp deviation inequalities are available for $\Ncal_t$. Applying Bernstein's inequality leads to
\begin{equation}\label{bernstein_covariance}
 \P\big(\big|\Ncal_t-\esp[\Ncal_t]\big|\geqslant u\big)\leqslant 2\exp \Big
(-\frac{u^2}{2\sigma_t^2+u}\Big),
\end{equation}
where $\sigma_t^2$ is the variance of $\Ncal_t$ (see for example \cite{Va_1998_asymptotic_statistics}). Götze and Tikhomirov proved in \cite{GoTi_2004_convergence_covariance} that, as soon as $1 < A_1 \leqslant \frac{m}{n} \leqslant A_2$, there exists a positive constant $C_1$ depending only on $A_1$ and $A_2$ such that
\begin{equation}\label{eq_gotze_tikho_covariance}
\sup_{t \in \R} \Big|\esp[\Ncal_t]-n\int_{-\infty}^t\mu_{m,n}(x)\,dx\Big| \leqslant C_1,
\end{equation}
for every LUE matrix $S_{m,n}$. For simplicity, we denote $\int_{-\infty}^t\mu_{m,n}(x)\,dx$ by $\mu_t$. Together with \eqref{bernstein_covariance}, for every $u \geqslant 0$,
\begin{equation}\label{eq_inequality_counting_function_lue}
 \P\big(|\Ncal_t-n\mu_t|\geqslant u+C_1\big)\leqslant
2\exp\Big(-\frac{u^2}{2\sigma_t^2+u}\Big).
\end{equation}
Among Su's results (see \cite{Su_2006_fluctuations}), for every $\delta >0$,
there exists $c_{\delta}>0$ such that
\begin{equation}\label{sup_variance_lue}
\sup_{t \in I_{\delta}} \sigma_t^2 \leqslant c_{\delta}\log n,
\end{equation}
where $I_{\delta}=[a_{m,n}+\delta,b_{m,n}-\delta]$. Combining \eqref{eq_inequality_counting_function_lue} and \eqref{sup_variance_lue}, deviation inequalities for individual eigenvalues in the bulk are then available, as stated in the following proposition.
%
\begin{prop}\label{prop_deviation_lue_bulk}
Assume that $1 < A_1 \leqslant \frac{m}{n} \leqslant A_2$. Let $\eta>0$.
Then there exist positive constants $C$, $c$, $c'$ and $\delta$ such that the following holds. For any LUE matrix $S_{m,n}$, for all $\eta n \leqslant j \leqslant (1-\eta)n$, for all $c \leqslant u \leqslant c'n$,
\begin{equation}\label{eq_deviation_lue_bulk}
\P\big(|\lambda_j-\gamma_j^{m,n}| \geqslant \frac{u}{n}\big) \leqslant 4 \exp\Big(-\frac{C^2u^2}{2c_{\delta}\log n+Cu}\Big).
\end{equation}
The constants $C$, $c'$ and $\delta$ depend only on $\eta$ and $A_2$, whereas the constant $c$ depends only on $\eta$, $A_1$ and $A_2$.
\end{prop}
\noindent Note that this proposition still holds if $m=n$, as Götze and Tikhomirov proved in \cite{GoTi_2004_convergence_covariance} that \eqref{eq_gotze_tikho_covariance} holds in that case.
%
\begin{proof}
Let $\eta>0$ and $u \geqslant 0$. Assume first that $\frac{n}{2} \leqslant j \leqslant (1-\eta)n$. Start with estimating the probability that $\lambda_j$ is greater than $\gamma_j^{m,n}+\frac{u}{n}$.
\begin{align*}
\P\Big(\lambda_j >\gamma_j^{m,n}+ \frac{u}{n}\Big) & = \P\big(\Ncal_{\gamma_j^{m,n}+\frac{u}{n}}<j\big)\\
 & = \P\big(n\mu_{\gamma_j^{m,n}+\frac{u}{n}}-\Ncal_{\gamma_j^{m,n}+\frac{u}{n}}>n\mu_{\gamma_j^{m,n}+\frac{u}{n}}-j\big)\\
 & =
\P\big(n\mu_{\gamma_j^{m,n}+\frac{u}{n}}-\Ncal_{\gamma_j^{m,n}+\frac{u}{n}}>n(\mu_{\gamma_j^{m,n}+\frac{u}{n}}-\mu_{\gamma_j^{m,n}} )\big)\\
 & \leqslant \P\big(|\Ncal_{\gamma_j^{m,n}+\frac{u}{n}}-n\mu_{\gamma_j^{m,n}+\frac{u}{n}}|>n(\mu_{\gamma_j^{m,n}+\frac{u}{n}}-\mu_{\gamma_j^{m,n}} )\big),
\end{align*}
where it has been used that $\mu_{\gamma_j^{m,n}}=\frac{j}{n}$. In order to use \eqref{eq_inequality_counting_function_lue}, a lower bound on $n(\mu_{\gamma_j^{m,n}+\frac{u}{n}}-\mu_{\gamma_j^{m,n}})$ is needed.
\begin{align*}
\mu_{\gamma_j^{m,n}+\frac{u}{n}}-\mu_{\gamma_j^{m,n}} & = \int_{\gamma_j^{m,n}}^{\gamma_j^{m,n}+\frac{u}{n}} \frac{1}{2\pi x}\sqrt{(b_{m,n}-x)(x-a_{m,n})}\,dx\\
 & \geqslant \frac{\sqrt{\gamma_j^{m,n}-a_{m,n}}}{2\pi b_{m,n}}\int_{\gamma_j^{m,n}}^{\gamma_j^{m,n}+\frac{u}{n}} \sqrt{b_{m,n}-x}\,dx\\
   & \geqslant \frac{\sqrt{\gamma_j^{m,n}-a_{m,n}}}{3\pi b_{m,n}}(b_{m,n}-\gamma_j^{m,n})^{3/2} \Big(1-\big(1-\frac{u/n}{b_{m,n}-\gamma_j^{m,n}}\big)^{3/2}\Big)\\
    & \geqslant \frac{\sqrt{\gamma_j^{m,n}-a_{m,n}}}{3\pi b_{m,n}}(b_{m,n}-\gamma_j^{m,n})^{1/2}\frac{u}{n},
\end{align*}
if $\gamma_j^{m,n}+\frac{u}{n} \leqslant b_{m,n}$. Furthermore, by definition of $\gamma_j^{m,n}$,
\begin{align*}
1-\frac{j}{n} & = \int_{\gamma_j^{m,n}}^{b_{m,n}}\frac{1}{2\pi
x}\sqrt{(x-a_{m,n})(b_{m,n}-x)}\,dx\\
 & \leqslant \frac{\sqrt{b_{m,n}-a_{m,n}}}{2\pi\gamma_j^{m,n}}
\int_{\gamma_j^{m,n}}^{b_{m,n}}\sqrt{b_{m,n}-x}\,dx\\
 & \leqslant 
\frac{\sqrt{b_{m,n}-a_{m,n}}}{3\pi\gamma_j^{m,n}}\big(b_{m,n}-\gamma_j^{m,n}\big)^{3/2}.
\end{align*}
Then
\begin{equation}\label{eq_distance_gamma_bord_droit_dependant_j}
b_{m,n}-\gamma_j^{m,n} \geqslant \Big(\frac{3\pi}{\sqrt{b_{m,n}-a_{m,n}}}\gamma_j^{m,n}\big(1-\frac{j}{n}\big)\Big)^{2/3}.
\end{equation}
Moreover
\begin{align*}
\frac{j}{n} & = \int_{a_{m,n}}^{\gamma_j^{m,n}}\frac{1}{2\pi
x}\sqrt{(x-a_{m,n})(b_{m,n}-x)}\,dx\\
 & \leqslant \frac{\sqrt{b_{m,n}-a_{m,n}}}{2\pi}
\int_{a_{m,n}}^{\gamma_j^{m,n}}\frac{1}{x}\sqrt{x-a_{m,n}}\,dx\\
 & \leqslant \frac{\sqrt{b_{m,n}-a_{m,n}}}{2\pi}\int_0^{\sqrt{\gamma_j^{m,n}-a_{m,n}}}\frac{2v^2}{v^2+a_{m,n}}\,dv,
\end{align*}
with change of variables $x=v^2+a_{m,n}$. Therefore,
\[ \frac{j}{n} \leqslant \frac{\sqrt{b_{m,n}-a_{m,n}}}{\pi}\sqrt{\gamma_j^{m,n}-a_{m,n}}.\]
Then 
\begin{equation}\label{eq_distance_gamma_bord_gauche}
\sqrt{\gamma_j^{m,n}} \geqslant \sqrt{\gamma_j^{m,n}-a_{m,n}} \geqslant
\frac{\pi}{\sqrt{b_{m,n}-a_{m,n}}}\frac{j}{n} \geqslant \frac{\pi}{2\sqrt{b_{m,n}-a_{m,n}}}, 
\end{equation}
as $j \geqslant \frac{n}{2}$.
Therefore, as $1-\frac{j}{n}\geqslant \eta$,
\begin{equation}\label{eq_distance_gamma_bord_droit}
b_{m,n}-\gamma_j^{m,n} \geqslant \big(\tfrac{3}{4}\big)^{2/3}\frac{\pi^2}{b_{m,n}-a_{m,n}}\eta^{2/3}.
\end{equation}
As a consequence, a lower bound on $\mu_{\gamma_j^{m,n}+\frac{u}{n}}-\mu_{\gamma_j^{m,n}}$ is achieved.
\begin{align*}
\mu_{\gamma_j^{m,n}+\frac{u}{n}}-\mu_{\gamma_j^{m,n}} & \geqslant \frac{C}{b_{m,n}(b_{m,n}-a_{m,n})}\eta^{1/3}\frac{u}{n},
\end{align*}
where $C>0$ is a universal constant.
As $\frac{m}{n} \leqslant A_2$,
\[\mu_{\gamma_j^{m,n}+\frac{u}{n}}-\mu_{\gamma_j^{m,n}} \geqslant \frac{C}{
4\sqrt{A_2}(1+\sqrt{A_2})^2}\eta^{1/3}\frac{u}{n}=C(A_2,\eta)\frac{u}{n}. \]
Then
\[\P\Big(\lambda_j> \gamma_j^{m,n}+\frac{u}{n}\Big)  \leqslant 
\P\big(|\mathcal{N}_{\gamma_j^{m,n}+\frac{u}{n}}-n\mu_{\gamma_j^{m,n}+\frac{u}{n}}|>C(A_2,\eta)u\big). \]
This is true for all $u \leqslant n(b_{m,n}-\gamma_j^{m,n})$. From \eqref{eq_distance_gamma_bord_droit}, this will be true when $u \leqslant c'n$ where $c'>0$ depends only on $A_2$ and $\eta$.
If $u \geqslant c=\frac{2C_1}{C(A_2,\eta)}$, then 
\[ \P\Big(\lambda_j> \gamma_j^{m,n}+\frac{u}{n}\Big)  \leqslant \P\big(|\mathcal{N}_{\gamma_j^{m,n}+\frac{u}{n}}-n\mu_{\gamma_j^{m,n}+\frac{u}{n}}|>\tfrac{1}{2}C(A_2,\eta)u+C_1\big).\]
Consequently, from \eqref{eq_inequality_counting_function_lue}, we get
\[\P\big(\lambda_j>\gamma_j^{m,n}+\frac{u}{n} \big)  \leqslant 2\exp\Big(-\frac{u^2}{2\sigma_{\gamma_j^{m,n}+\frac{u}{n}}^2+u}\Big). \]
For $u \leqslant c'n$ (with a maybe smaller $c'>0$ depending only on $\eta$ and $A_2$), there exists $\delta>0$ depending on $\eta$ and $A_2$ such that $\gamma_j^{m,n}+\frac{u}{n} \in I_{\delta}$. Consequently, from \eqref{sup_variance_lue}, $\sigma_{\gamma_j^{m,n}+\frac{u}{n}}^2 \leqslant c_{\delta}\log n$. Then, for $c \leqslant u \leqslant c'n$,
\[\P\big(\lambda_j> \gamma_j^{m,n}+\frac{u}{n}\big)  \leqslant 2\exp\Big(-\frac{u^2}{2c_{\delta}\log n+u}\Big). \]
Repeating the argument leads to the same bound on $\P\big(\lambda_j <\gamma_j^{m,n}- \frac{u}{n}\big)$.
Therefore,
\begin{equation*}
\P\Big(|\lambda_j-\gamma_j^{m,n}|\geqslant \frac{u}{n}\Big)\leqslant 4\exp\Big(-\frac{C^2u^2
}{2c_{\delta}\log n+Cu}\Big).
\end{equation*}
The case when $j \leqslant \frac{n}{2}$ is treated similarly.
The proposition is thus established.

\end{proof}
We turn now to the proof of Theorem \ref{thm_LUE_bulk}.
\begin{proof}[Proof of Theorem \ref{thm_LUE_bulk}]
Note first that, for every $i$,
\begin{equation*}
\esp[\lambda_i^{4}] \leqslant \sum_{j=1}^n\esp[\lambda_j^{4}]=\esp[\Tr(S_{m,n}^4)].
\end{equation*}
From Hölder inequality,
\begin{equation}\label{trace_bound}
 \esp[\Tr(S_{m,n}^4)] \leqslant \frac{1}{n^4} \sum_{\substack{j_1,\dots,j_4 \in [\![1,m]\!]\\ i_1,\dots,i_4 [\![1,n]\!]}} \big(\esp[|X_{i_1,j_1}|^8]\big)^{1/8} \dots
\big(\esp[|X_{i_4,j_1}|^8]\big)^{1/8}.
\end{equation}
As $S_{m,n}$ is from the LUE, the $8^{\textrm{th}}$ moment of its entries is $\esp[|X_{i,j}|^8]=105$. Then
\[\esp[\Tr(S_{m,n}^4)] \leqslant 105m^4 \leqslant 105A_2^4n^4. \]
Consequently, for all $n \geqslant 1$, for all $1 \leqslant i \leqslant n$,
\begin{equation}\label{bound_moment_LUE}
\esp[\lambda_i^{4}] \leqslant 105A_2^4n^4.
\end{equation}
Consider constants $C$, $c$, $c'$ and $\delta$ given by Proposition \ref{prop_deviation_lue_bulk}.
Choose next $M>0$ large enough such that
$\frac{C^2M^2}{2c_{\delta}+CM}>8$. M depends only on $\eta$ and $A_2$. Setting $Z=n|\lambda_j-\gamma_j^{m,n}|$,
\begin{align*}
\esp[Z^2] & = \int_0^{\infty}\P(Z\geqslant v)2v\,dv\\
 & = \int_0^{c}\P(Z\geqslant v)2v\,dv + \int_{c}^{M\log
n}\P(Z\geqslant v)2v\,dv + \int_{M\log n}^{\infty}\P(Z\geqslant v)2v\,dv\\
 & \leqslant c^2 + I_1 + I_2.
\end{align*}
The two latter integrals are handled in different ways. The first one $I_1$
is bounded using \eqref{eq_deviation_lue_bulk} while $I_2$ is controlled using 
the Cauchy-Schwarz inequality and \eqref{bound_moment_LUE}.
Starting thus with $I_2$,
\begin{align*}
I_2 & = \int_{M\log n}^{+\infty}\P(Z\geqslant v)2v\,dv\\
 & \leqslant \esp\big[Z^2\mathbbm{1}_{Z \geqslant M\log n}\big]\\
 & \leqslant \sqrt{\esp\big[Z^{4}\big]} \sqrt{\P\big(Z \geqslant M\log n\big)}\\
 & \leqslant An^{4}\sqrt{\P\Big(|\lambda_j-\gamma_j^{m,n}|\geqslant \frac{M\log n}{n}\Big)}\\
 & \leqslant 2An^{4}\exp\Big(-\frac{1}{2}\frac{C^2M^2}{2c_{\delta}+CM}\log n\Big)\\
 & \leqslant 2A
\exp\bigg(\frac{1}{2}\Big(8-\frac{C^2M^2}{2c_{\delta}+CM}\Big)\log n\bigg),
\end{align*}
where $A>0$ is a numerical constant. As
$\exp\Big(\frac{1}{2}\big(8-\frac{C^2M^2}{2c_{\delta}+CM}\big)\log n\Big)
\underset{n \to
\infty}{\to} 0$, there exists a constant $C>0$ (depending only on $\eta$ and $A_2$) such that \[I_2 \leqslant
C.\]
Turning to $I_1$, recall that Proposition \ref{prop_deviation_lue_bulk} gives, for $c \leqslant v
\leqslant c'n$, 
\[P(Z\geqslant v) =
\P\Big(|\lambda_j-\gamma_j^{m,n}|\geqslant \frac{v}{n}\Big) \leqslant 4\exp\Big(-\frac{C^2v^2
}{2c_{\delta}\log
n+Cv}\Big).\] 
Hence in the range $v \leqslant M\log n$, 
\[P(Z\geqslant v) \leqslant
4\exp\Big(-\frac{B}{\log n}v^2\Big),\]
where $B=B(A_2,\eta)=\frac{C^2}{2c_{\delta}+CM}$.
As a consequence,
 \[I_1 \leqslant 4\int_{c}^{M\log
n}\exp\Big(-\frac{B}{\log n}v^2\Big)2v\,dv \leqslant
\frac{4\log n}{B}\int_0^{\infty}e^{-v^2}2v\,dv.\]
There exists thus a constant $C>0$ (depending only on $\eta$ and $A_2$) such that 
\[I_1 \leqslant C\log n.\]
Summarizing the previous steps, $\esp\big[Z^2\big] \leqslant C\log n$. Therefore
\[\esp\big[|\lambda_j-\gamma_j^{m,n}|^2\big]\leqslant
C\frac{\log n}{n^2}, \]
$C$ depending only on $A_1$, $A_2$ and $\eta$, which is the claim. The proof of Theorem \ref{thm_LUE_bulk} is complete.
\end{proof}

\subsection{Between the bulk and the edge of the spectrum}\label{section_intermediate_LUE}
The aim of this section is to prove an analogous theorem for some eigenvalues between the bulk and the right edge of the spectrum, i.e. for $\lambda_j$ such that $(1-\eta)n \leqslant j \leqslant n-K\log n$. The precise statement is the following.
\begin{thm_e}\label{thm_LUE_intermediate}
There exists a constant $\kappa>0$ (depending on $A_1$ and $A_2$) such that the following holds. For all $K>\kappa$, for all $0<\eta \leqslant \frac{1}{2}$, there exists a constant $C>0$~(depending on $K$, $\eta$, $A_1$ and $A_2$) such that for all covariance matrix $S_{m,n}$, for all $(1-\eta)n \leqslant j \leqslant n-K\log n$,
\begin{equation}
\esp[(\lambda_j-\gamma_j^{m,n})^2] \leqslant C\frac{\log (n-j)}{n^{4/3}(n-j)^{2/3}}.
\end{equation}
In particular,
\begin{equation}
 \Var(\lambda_j) \leqslant C\frac{\log (n-j)}{n^{4/3}(n-j)^{2/3}}.
\end{equation}
\end{thm_e}
As for eigenvalues in the bulk, the proof relies on the determinantal structure of LUE matrices. Recall that this structure together with a bound on the mean counting function \eqref{eq_gotze_tikho_covariance} leads to the following deviation inequality for the counting function $\Ncal_t$.
\[ \P\big(|\Ncal_t-n\mu_t|\geqslant u+C_1\big)\leqslant
2\exp\Big(-\frac{u^2}{2\sigma_t^2+u}\Big).\]
Among Su's results (see \cite{Su_2006_fluctuations}), for every $\tilde{\delta} >0$, for every $\tilde{K}>0$,
there exists $c_{\tilde{\delta},\tilde{K}}>0$ such that for all $t$ satisfying $0< b_{m,n}-t < \tilde{\delta}$ and $n(b_{m,n}-t)^{3/2} \geqslant \tilde{K}\log n$,
\begin{equation}\label{sup_variance_lue_interm}
 \sigma_t^2 \leqslant c_{\tilde{\delta},\tilde{K}}\log n(b_{m,n}-t)^{3/2}.
\end{equation}
Combining \eqref{eq_inequality_counting_function_lue} and \eqref{sup_variance_lue_interm}, deviation inequalities for individual intermediate eigenvalues are then available.
\begin{prop}\label{prop_deviation_lue_interm}
Assume that $1 < A_1 \leqslant \frac{m}{n} \leqslant A_2$. There exists $\kappa>0$ depending only on $A_1$ and $A_2$ such that the following holds. Let $K>\kappa$ and $0< \eta \leqslant \frac{1}{2}$.
Then there exist positive constants $C$, $c$, $C'$ and $c'$ such that the following holds. For any LUE matrix $S_{m,n}$, for all $(1-\eta) n \leqslant j \leqslant n-K\log n$, for all $c \leqslant u \leqslant c'n$,
\begin{equation}\label{eq_deviation_lue_interm}
\P\big(|\lambda_j-\gamma_j^{m,n}| \geqslant \frac{u}{n^{2/3}(n-j)^{1/3}}\big) \leqslant 4 \exp\Big(-\frac{C^2u^2}{C'\log (n-j)+Cu}\Big).
\end{equation}
The constants $C$, $C'$ and $c'$ depend only on $K$, $\eta$ and $A_2$, whereas the constant $c$ depends only on $K$, $\eta$, $A_1$ and $A_2$.
\end{prop}
Note that this proposition still holds when  $m=n$, for eigenvalues on the right-side of the spectrum. The proof of Proposition \ref{prop_deviation_lue_interm} is very similar to what was done for eigenvalues in the bulk. Therefore some details are not reproduced.
\begin{proof}
Let $\eta>0$, $K>0$ and $u \geqslant 0$. Assume that $(1-\eta)n\leqslant j \leqslant n-K\log n$. Set $u_{n,j}=\frac{u}{n^{2/3}(n-j)^{1/3}}$. As for the bulk case, we start with estimating the probability that $\lambda_j$ is greater than $\gamma_j^{m,n}+u_{n,j}$. We get
\begin{equation*}
\P\big(\lambda_j >\gamma_j^{m,n}+ u_{n,j}\big)  \leqslant
\P\big(|\Ncal_{\gamma_j^{m,n}+u_{n,j}}-n\mu_{\gamma_j^{m,n}+u_{n,j}}|> n(\mu_{\gamma_j^{m,n}+u_{n,j}}-\mu_{\gamma_j^{m,n}} )\big).
\end{equation*}
Furthermore,
\begin{align*}
 \mu_{\gamma_j^{m,n}+u_{n,j}}-\mu_{\gamma_j^{m,n}} 
    & \geqslant \frac{\sqrt{\gamma_j^{m,n}-a_{m,n}}}{3\pi b_{m,n}}(b_{m,n}-\gamma_j^{m,n})^{1/2}u_{n,j},
\end{align*}
if $\gamma_j^{m,n}+u_{n,j} \leqslant b_{m,n}$. From \eqref{eq_distance_gamma_bord_droit_dependant_j},
\begin{equation*}
b_{m,n}-\gamma_j^{m,n} \geqslant \Big(\frac{3\pi}{\sqrt{b_{m,n}-a_{m,n}}}\gamma_j^{m,n}\big(\frac{n-j}{n}\big)\Big)^{2/3}.
\end{equation*}
Moreover, as $\eta \leqslant \frac{1}{2}$, from \eqref{eq_distance_gamma_bord_gauche},
\[\sqrt{\gamma_j^{m,n}} \geqslant \sqrt{\gamma_j^{m,n}-a_{m,n}} \geqslant
\frac{\pi}{2\sqrt{b_{m,n}-a_{m,n}}}. \]
Therefore
\begin{equation}\label{eq_distance_gamma_bord_droit_interm}
b_{m,n}-\gamma_j^{m,n} \geqslant \big(\tfrac{3}{4}\big)^{2/3}\frac{\pi^2}{b_{m,n}-a_{m,n}}\big(\frac{n-j}{n}\big)^{2/3},
\end{equation}
and
\begin{align*}
\mu_{\gamma_j^{m,n}+u_{n,j}}-\mu_{\gamma_j^{m,n}} & \geqslant \frac{C}{b_{m,n}(b_{m,n}-a_{m,n})}\frac{u}{n},
\end{align*}
where $C>0$ is a universal constant.
As $\frac{m}{n} \leqslant A_2$,
\[\mu_{\gamma_j^{m,n}+u_{n,j}}-\mu_{\gamma_j^{m,n}} \geqslant \frac{C}{
4\sqrt{A_2}(1+\sqrt{A_2})^2}\frac{u}{n}=C(A_2)\frac{u}{n}. \]
Similarly to the bulk case, we get
\[\P(\lambda_j>\gamma_j^{m,n}+u_{n,j} )  \leqslant 2\exp\Big(-\frac{u^2}{2\sigma_{\gamma_j^{m,n}+u_{n,j}}^2+u}\Big). \]
This relation holds if $c \leqslant u \leqslant n^{2/3}(n-j)^{1/3}(b_{m,n}-\gamma_j^{m,n})$, with $c$ depending only on $A_1$ and $A_2$. Let $\alpha \in (0,1)$. Set $c'=\alpha(\frac{3}{4})^{2/3}\frac{\pi^2}{4\sqrt{A_2}}$, depending only on $\alpha$ and $A_2$. If $u \leqslant c'(n-j)$, then, due to \eqref{eq_distance_gamma_bord_droit_interm}, the preceding relation holds. The bound \eqref{sup_variance_lue_interm} on $\sigma_{t_n}^2$ obtained by Su holds when $0<b_{m,n}-t_n \leqslant \tilde{\delta}$ and $n(b_{m,n}-t_n)^{3/2} \geqslant \tilde{K}\log n$. Set
$t_n=\gamma_j^{m,n}+u_{n,j}$. As $u \geqslant 0$, $0 < b_{m,n}-t_n \leqslant b_{m,n}-\gamma_j^{m,n}$. Therefore, as
$j \geqslant (1-\eta)n$, similar computations as for \eqref{eq_distance_gamma_bord_droit} lead to
\[b_{m,n}-t_n \leqslant
\Big(\frac{3b_{m,n}\sqrt{b_{m,n}-a_{m,n}}}{1-\eta}\eta\Big)^{2/3} \leqslant  \Big(\frac{6(A_2)^{1/4}(1+\sqrt{A_2})^2}{1-\eta}\eta\Big)^{2/3}=\tilde{\delta}\] for
all $n$. Moreover,
\begin{align*}
n(b_{m,n}-t_n)^{3/2} & = n(b_{m,n}-\gamma_j^{m,n})^{3/2}\Big(1-\frac{u_{n,j}}{b_{m,n}-\gamma_j^{m,n}}\Big)^{3/2}\\
  & \geqslant \frac{3\pi^3}{4(b_{m,n}-a_{m,n})^{3/2}}(n-j)\Big(1-\frac{c'(n-j)^{2/3}}{n^{2/3}(b_{m,n}-\gamma_j^{m,n})}\Big)^{3/2}\\
 & \geqslant \frac{3\pi^3}{4(4\sqrt{A_2})^{3/2}}(1-\alpha)^{3/2}K\log n\\
 & \geqslant \tilde{K}\log n,
\end{align*}
where $\tilde{K}=\frac{3\pi^3}{4(4\sqrt{A_2})^{3/2}}(1-\alpha)^{3/2}K>0$.
From \eqref{sup_variance_lue_interm}, for all $c \leqslant u \leqslant c'(n-j)$,
\[\Var(\Ncal_{\gamma_j^{m,n}+ u_{n,j}}) 
    \leqslant c_{\tilde{\eta},\tilde{K}}\log\big(n(b_{m,n}-t_n)^{3/2}\big).\] But
\begin{align*}
n(b_{m,n}-t_n)^{3/2} & \leqslant n\big(b_{m,n}-\gamma_j^{m,n}\big)^{3/2}.
\end{align*}
Using the same techniques as for \eqref{eq_distance_gamma_bord_droit_dependant_j}, it is possible to show that
\[(b_{m,n}-\gamma_j^{m,n})^{3/2} \leqslant 6b_{m,n}\sqrt{b_{m,n}-a_{m,n}}\frac{n-j}{n} \leqslant 12(A_2)^{1/4}\big(1+\sqrt{A_2}\big)^2\frac{n-j}{n}. \]
Hence $\log\big(n(b_{m,n}-t_n)^{3/2}\big) 
\leqslant \log(n-j)+\log(12(A_2)^{1/4}(1+\sqrt{A_2})^2)$. For $K>\kappa$ with $\kappa$ large enough depending only on $A_2$ and for $n \geqslant 2$, $n-j \geqslant K\log n \geqslant 12(A_2)^{1/4}(1+\sqrt{A_2})^2$ and $\Var(\Ncal_{\gamma_j^{m,n}+
u_{n,j}}) \leqslant 2c_{\tilde{\delta},\tilde{K}}\log(n-j)$.
Therefore
\[\P\big(\lambda_j >\gamma_j^{m,n}+ u_{n,j}\big) \leqslant 2\exp\Big(-\frac{C^2u^2
}{4c_{\tilde{\delta},\tilde{K}}\log(n-j)+Cu}\Big).\]
The proof is concluded similarly to Proposition \ref{prop_deviation_lue_bulk}.

\end{proof}
We turn now to the proof of Theorem \ref{thm_LUE_intermediate}, in which some details are skipped, due to the similarity with the proof of Theorem \ref{thm_LUE_bulk}.
\begin{proof}[Proof of Theorem \ref{thm_LUE_intermediate}]
Setting $Z=n^{2/3}(n-j)^{1/3}|\lambda_j-\gamma_j^{m,n}|$,
\begin{align*}
\esp[Z^2] & = \int_0^{\infty}\P(Z\geqslant v)2v\,dv\\
 & = \int_0^{c}\P(Z\geqslant v)2v\,dv + \int_{c}^{\frac{C'}{C}\log(n-j)}\P(Z\geqslant v)2v\,dv\\
 & +
\int_{\frac{C'}{C}\log(n-j)}^{c'(n-j)}\P(Z\geqslant v)2v\,dv + \int_{c'(n-j)}^{\infty}\P(Z\geqslant
v)2v\,dv\\
 & \leqslant c^2 + J_1 + J_2+J_3,
\end{align*}
where $c$, $c'$, $C$ and $C'$ are given by Proposition \ref{prop_deviation_lue_interm}.
Repeating the computations carried out with $I_2$ in the proof of Theorem \ref{thm_LUE_bulk} yields
\begin{align*}
J_3  & \leqslant 2n^{4/3}(n-j)^{2/3} \sqrt{\esp[(\lambda_j-\gamma_j^{m,n})^4]}\exp\Big(-\frac{1}{2} \frac{C^2C'^2(n-j)^2}{C'\log(n-j)+Cc'(n-j)}\Big)\\
 & \leqslant 2An^{4}\exp\Big(-\frac{1}{2}\frac{C^2c'^2(n-j)^2}{C'\log(n-j)+Cc'(n-j)}\Big),
\end{align*}
where $A>0$ is a numerical constant. The last inequality is due to \eqref{bound_moment_LUE}. For $n$ large enough (depending on $\eta$, $A_2$ and $K$), $C'\log(n-j)
\leqslant Cc'(n-j)$ and
\[ J_3 \leqslant
2A\exp\Big(4\log n-\frac{Cc'}{4}(n-j)\Big).\]
Then, as $n-j \geqslant K\log n$,
\[J_3 \leqslant 2A\exp\Big(\big(4-\frac{KCc'}{4}\big)\log n\Big).\] Recall from the proof of Proposition \ref{prop_deviation_lue_interm} that $c'=\alpha c'(A_2)$ where $\alpha \in (0,1)$ is a universal constant and $c'(A_2)$ depends only on $A_2$. Furthermore, the constant $C$ depends only on $A_2$. Therefore, if we choose $\kappa>0$ such that $\kappa>\frac{16}{Cc'(A_2)}$, then $\frac{KCc'}{4} > 4$. The right-hand side goes thus to $0$ when $n$ goes to infinity. As a consequence, there exists a constant $C>0$ depending only on $A_2$, $\eta$ and $K$ such that \[J_3 \leqslant C.\]
The integral $J_1$ is handled as $I_1$, using that, in the range $v \leqslant \frac{C'}{C}\log(n-j)$, 
 \[P(Z\geqslant v) \leqslant
 4\exp\Big(-\frac{B}{\log (n-j)}v^2\Big),\]
 where $B$ depends only on $K$, $\eta$ and $A_2$ (this is due to Proposition
 \ref{prop_deviation_lue_interm}).
Hence, there exists a constant $C$ depending only on $A_2$, $\eta$ and $K$ such that
\[J_1 \leqslant C\log(n-j).\]
Finally, $J_2$ is handled similarly. In the
range $\frac{C'}{C}\log(n-j) \leqslant v \leqslant c'(n-j)$, from Proposition \ref{prop_deviation_lue_interm},
\[P(Z\geqslant v) \leqslant 4\exp\Big(-\frac{C}{2}v\Big).\]
Thus
\[J_2 \leqslant 4\int_{\frac{C'}{C}\log(n-j)}^{c'(n-j)}\exp\Big(-\frac{C}{2}v\Big)2v\,dv \leqslant
4\int_{0}^{\infty}\exp\Big(-\frac{C}{2}v\Big)2v\,dv.\]
Then $J_2$ is bounded by a constant, which depends only on $A_2$. 
There exists thus a constant $C>0$ such that 
\[J_2 \leqslant C.\]
Summarizing the previous steps, $\esp[Z^2] \leqslant C\log (n-j)$, where $C$ depends only on $A_1$, $A_2$, $\eta$ and $K$. Therefore
\[\esp\big[|\lambda_j-\gamma_j|^2\big]\leqslant
C\frac{\log (n-j)}{n^{4/3}(n-j)^{2/3}}, \]
which is the claim.
\end{proof}

\subsection{At the edge of the spectrum}\label{section_LUE_edge}
In \cite{LeRi_2010_deviations}, Ledoux and Rider gave unified proofs of precise small deviation inequalities
for the largest eigenvalues of $\beta$-ensembles. The results hold in particular for LUE matrices ($\beta=2$)
and for LOE matrices ($\beta=1$). The following theorem summarizes some of the relevant inequalities for the
LUE.

\begin{thm_e} \cite{LeRi_2010_deviations}
Let $A_1>1$. There exists a constant $C>0$ depending only on $A_1$ such that the following holds. Let $S_{m,n}$ be a LUE matrix. Denote by $\lambda_n$ the maximal eigenvalue of $S_{m,n}$. Then, for all $n \in \mathbb{N}$, for all $m \in \mathbb{N}$ such that $m>A_1n$ and for all $0< \varepsilon\leqslant 1$,
\begin{equation}
 \P\big(\lambda_{n} \leqslant b_{m,n}(1-\varepsilon)\big) \leqslant C^2\exp\Big(-\frac{2
}{C}n^2\varepsilon^3\Big),
\end{equation}
and
\begin{equation}
 \P\big(\lambda_{n} \geqslant b_{m,n}(1+\varepsilon)\big) \leqslant
C\exp\Big(-\frac{2}{C}n\varepsilon^{3/2}\Big).
\end{equation}
\end{thm_e}                               
The large deviation tails are also known. The following corollary can be deduced by integrating these inequalities. 

\begin{cor}\label{thm_LUE_edge}\cite{LeRi_2010_deviations}
Let $S_{m,n}$ be a LUE matrix. Then there exists a universal
 constant $C>0$ such that for all $ n \geqslant 1$, for all $m \in \mathbb{N}$ such that $m>A_1n$,
\[ \Var(\lambda_{n}) \leqslant \esp\big[(\lambda_n-b_{m,n})^2\big] \leqslant  Cn^{-4/3}.\]
\end{cor}
Similar results are probably true for the $k^{\textrm{th}}$ largest eigenvalue (for $k \in
\mathbb{N}$ fixed). The authors established also a left-side deviation inequality for the smallest eigenvalue in the case when $m>A_1n$.
\begin{equation}
 \P\big(\lambda_{1} \leqslant a_{m,n}(1-\varepsilon)\big) \leqslant
C\exp\Big(-\frac{2}{C}n\varepsilon^{3/2}\Big),
\end{equation}
for all $0< \varepsilon\leqslant 1$. But no right-side deviation inequality seems to be known for the smallest eigenvalue $\lambda_1$ and therefore we cannot deduce a precise bound on the variance of the smallest eigenvalue.

%


\section[Bounds for covariance matrices]{Variance bounds for families of covariance matrices}\label{section_covariance}
The previously achieved bounds on the variance of eigenvalues for LUE matrices are then extended to families of more general covariance matrices. It is due to the combination of two very recent results, some localization properties established by Pillai and Yin \cite{PiYi_2011_covariance} and to the Four Moment Theorem proved by Tao and Vu \cite{TaVu_2012_covariance} and Wang \cite{Wa_2011_covariance_edge}.

\subsection{Localization properties and the Four Moment Theorem}\label{tools}
This subsection is devoted to the statement of the previously mentioned results which will be used in order to extend variance bounds to large families of non Gaussian covariance matrices. Matrices which are considered in this section are covariance matrices $S_{m,n}$ satisfying condition $(C0)$, defined by the following.
Say that $S_{m,n}$ satisfies condition $(C0)$ if its entries $X_{ij}$ are
independent and have an exponential decay: there are positive constants $B_1$ and $B_2$ such that
\begin{equation*}
\forall \ i \in \{1,\dots,n\}, \forall \ j \in \{1,\dots,m\}, \ \P\big(|X_{ij}|\geqslant
t^{B_1}\big) \leqslant e^{-t}
\end{equation*}
for all $t \geqslant B_2$.

Pillai and Yin proved in \cite{PiYi_2011_covariance} a Localization Theorem similar to the one proved by Erdös, Yau and Yin in \cite{ErYaYi_2010_rigidity}. This theorem establishes that the eigenvalues are highly localized around their theoretical locations $\gamma_j^{m,n}$.
\begin{thm_e}[Localization \cite{PiYi_2011_covariance}]\label{thm_localization_covariance}
Let $S_{m,n}$ be a random covariance matrix whose entries satisfy condition $(C0)$. Suppose that $1<A_1 \leqslant
\frac{m}{n} \leqslant A_2<+\infty$.
There are positive universal constants $c$ and $C$ such that, for any $ 1 \leqslant j \leqslant n $,
\begin{equation}\label{eq_localization_inequality}
\P\Big(|\lambda_j-\gamma_j^{m,n}|\geqslant (\log n)^{C\log\log n}n^{-2/3}\min(j,n+1-j)^{-1/3}\Big)\leqslant
Ce^{-(\log
n)^{c\log\log n}}.
\end{equation}
\end{thm_e}
This deviation inequality \eqref{eq_localization_inequality} can be used to reach an almost optimal bound on the variance. Indeed, due to \eqref{eq_localization_inequality} and the Cauchy-Schwarz inequality, $\Var(\lambda_j)$ may be bounded by $\frac{(\log n)^{2C\log \log n}}{n^2}$ in the bulk of the spectrum, which is almost the right order for the variance. In order to remove the $\log\log n$ term, we turn now to the Four Moment Theorem. This theorem was proved for the bulk of the spectrum by Tao and Vu \cite{TaVu_2012_covariance} and extended to the edge by Wang \cite{Wa_2011_covariance_edge}. From now, we consider covariance matrices $S_{m,n}$ which satisfy condition $(C0)$ and whose entries match the entries of a LUE matrix up to order $4$. Say that two complex random variables $\xi$ and $\xi'$ match to order $k$ if
\begin{equation*}
\esp\big[ \Re(\xi)^m \Im(\xi)^l\big]
    =\esp\big[\Re(\xi')^m \Im(\xi')^l\big]
\end{equation*}
for all $m, l \geqslant 0$ such that $m+l \leqslant k$.

\begin{thm_e}[Four Moment Theorem \cite{TaVu_2012_covariance,Wa_2011_covariance_edge}]\label{thm_4_moments_cov}
There exists a small positive constant $c_0$ such that the following holds.
Let $S_{m,n}=\frac{1}{n}X^*X$ and $S_{m,n}'=\frac{1}{n}X'^*X'$ be
two random covariance matrices satisfying condition $(C0)$. Assume that, for $1\leqslant i \leqslant
n$, $X_{ij}$ and $X'_{ij}$
match to order $4$.
Let $G: \mathbb{R} \to \mathbb{R}$ be a smooth function such that:
\begin{equation}\label{eq_condition_derivees}
\forall \ 0 \leqslant k \leqslant 5, \quad \forall x \in \mathbb{R}, \quad \big|G^{(k)}(x)\big|
\leqslant
n^{c_0}.
\end{equation}\label{eq_four_moment_thm_covariance}
Then, for all $1 \leqslant i \leqslant n$ and for $n$ large
enough (depending on constants $B_1$ and $B_2$ in condition $(C0)$),
\begin{equation}
\big|\esp[G(n\lambda_{i})]-\esp[G(n\lambda_{i}')]\big|\leqslant n^{-c_0}.
\end{equation}
\end{thm_e}
Suppose Theorem \ref{thm_4_moments_cov} apply with $G_j: x \in \R \mapsto (x-n\gamma_j^{m,n})^2$. Then \eqref{eq_four_moment_thm_covariance} writes
\[ \big|n^2\esp[(\lambda_j-\gamma_j^{m,n})^2]-n^2\esp[(\lambda_j'-\gamma_j^{m,n})^2]\big|\leqslant n^{-c_0}. \]
As $\esp[(\lambda_j'-\gamma_j^{m,n})^2]$ is bounded by $\frac{\log n}{n^2}$, $\frac{\log (n-j)}{n^{4/3}(n-j)^{2/3}}$ or $n^{-4/3}$, which are bigger than $n^{-2-c_0}$, the bounds could be extended. Unfortunately, $G_j$ does not satisfy \eqref{eq_condition_derivees}. To get round this difficulty, the Four Moment Theorem \ref{thm_4_moments_cov} is applied to a smooth truncation of $G_j$. The Localization Theorem \ref{thm_localization_covariance} provides a small area around $\gamma_j^{m,n}$ where $\lambda_j$ is very likely to be in so that the error due to the truncation is well controlled. Details are contained in the following subsection.

\subsection{Comparison with LUE matrices}\label{comparison_covariance}
Let $S_{m,n}$ be a covariance matrix and $S_{m,n}'$ be a LUE matrix such that they satisfy the hypotheses of
Theorem \ref{thm_4_moments_cov}. Note that the following procedure is valid for eigenvalues in the bulk and at the edge of the spectrum, as well as for intermediate eigenvalues.

Let $1 \leqslant j \leqslant n$.
Set $R_n^{(j)}=(\log n)^{C\log\log
 n}n^{1/3}\min(j,n+1-j)^{-1/3}$ and $\varepsilon_n=Ce^{-(\log n)^{c\log\log n}}$. Then Theorem
\ref{thm_localization_covariance} leads to:
\begin{equation}\label{eq_short_localization_inequality}
\P\Big(|\lambda_j-\gamma_j^{m,n}| \geqslant \frac{R_n^{(j)}}{n}\Big) \leqslant \varepsilon_n. 
\end{equation}
Let $\psi$ be a smooth function with support $[-2,2]$ and values in $[0,1]$ such that
$\psi(x)=\frac{1}{10}x^2$ for all $x \in [-1; 1]$. Set $G_j: x \in \mathbb{R} \mapsto
\psi\big(\frac{x-n\gamma_j}{R_n^{(j)}}\big)$. We want to apply Tao and Vu's Four Moment Theorem
\ref{thm_4_moments_cov} to $G_j$. As $\psi$ is smooth and has compact support, its first five derivatives are
bounded by $M>0$. Then, for all $0 \leqslant k \leqslant 5$, for all $x \in \mathbb{R}$,
\begin{equation*}
\big|G_j^{(k)}(x)\big|\leqslant \frac{M}{(R_n^{(j)})^k} \leqslant n^{c_0},
\end{equation*}
where the last inequality holds for $n$ large enough (depending only on $M$ and $c_0$). Then, the Four Moment
Theorem \ref{thm_4_moments_cov} yields: 
\begin{equation}\label{eq_application_four_moment_thm_covariance}
\big|\esp[G_j(n\lambda_j)]-\esp[G_j(n\lambda_j')]\big| \leqslant n^{-c_0}
\end{equation}
for large enough $n$. But 
\begin{align*}
\esp[G_j(&n\lambda_j)] \\
& = \tfrac{1}{10}\esp\Big[\Big(\tfrac{n\lambda_j-n\gamma_j^{m,n}}{R_n^{(j)}}\Big)^2\mathbbm{1}_{\frac{
|n\lambda_j-n\gamma_j^{m,n}|}{R_n^{(j)}}\leqslant 1}\Big]
+\esp\Big[G_j(n\lambda_j)\mathbbm{1}_{\frac{
|n\lambda_j-n\gamma_j^{m,n}|}{R_n^{(j)}}> 1}\Big]\\
 & =
\tfrac{n^2}{10(R_n^{(j)})^2}\esp\Big[(\lambda_j-\gamma_j^{m,n})^2\mathbbm{1}_{|\lambda_j-\gamma_j^{m,n}|\leqslant
\frac{R_n^{(j)}}{n}}\Big]+\esp\Big[G_j(n\lambda_j)\mathbbm{1}_{\frac{
|n\lambda_j-n\gamma_j^{m,n}|}{R_n^{(j)}}> 1}\Big].\\
\end{align*}
On the one hand,
\begin{eqnarray*}
 \esp\Big[G_j(n\lambda_j)\mathbbm{1}_{\frac{
|n\lambda_j-n\gamma_j^{m,n}|}{R_n^{(j)}}> 1}\Big] & \leqslant &
\P\big(|n\lambda_j-n\gamma_j^{m,n}|>R_n^{(j)}\big)\\
 & \leqslant & \P\Big(|\lambda_j-\gamma_j^{m,n}|>\frac{R_n^{(j)}}{n}\Big)\\
 & \leqslant & \varepsilon_n.
\end{eqnarray*}
On the other hand,
\begin{equation*}
\esp\Big[(\lambda_j-\gamma_j^{m,n})^2\mathbbm{1}_{
|\lambda_j-\gamma_j^{m,n}|\leqslant\frac{R_n^{(j)}}{n}}\Big] = \esp\big[(\lambda_j-\gamma_j^{m,n})^2\big] -
\esp\Big[(\lambda_j-\gamma_j^{m,n})^2\mathbbm{1}_{|\lambda_j-\gamma_j^{m,n}|>\frac{R_n^{(j)}}{n}}\Big].
\end{equation*}
As condition $(C0)$ is satisfied, the $8$-th moment of the entries is uniformly bounded by a
constant which depends only on constants $B_1$ and $B_2$ in condition $(C0)$. Then, from the Cauchy-Schwarz
inequality and \eqref{trace_bound},
\begin{eqnarray*}
\esp\Big[(\lambda_j-\gamma_j^{m,n})^2\mathbbm{1}_{|\lambda_j-\gamma_j^{m,n}|>\frac{R_n^{(j)}}{n}}\Big] & \leqslant &
\sqrt{\esp\big[(\lambda_j-\gamma_j^{m,n})^4\big]\P\Big(|\lambda_j-\gamma_j^{m,n}|>\tfrac{R_n^{(j)}}{n}\Big)}\\
 & \leqslant &  An^2\sqrt{\varepsilon_n}
\end{eqnarray*}
where $A>0$ is a numerical constant. Then
\begin{eqnarray*}
\esp\big[G_j(n\lambda_j)\big] & = &
\frac{n^2}{10\big(R_n^{(j)}\big)^2}\Big(\esp\big[(\lambda_j-\gamma_j^{m,n})^2\big]+O\big(n^{2}\varepsilon_n^{1/2}
\big)\Big)+O(\varepsilon_n)\\
 & = &
\frac{n^2}{10\big(R_n^{(j)}\big)^2}\esp\big[(\lambda_j-\gamma_j^{m,n})^2\big]+O\big(n^{4}\varepsilon_n^{1/2}
\big(R_n^{(j)}\big)^{-2}
\big)+O(\varepsilon_n).\\
\end{eqnarray*}
Repeating the same computations gives similarly
\begin{equation*}
\esp\big[G_j(n\lambda_j')\big] = 
\frac{n^2}{10\big(R_n^{(j)}\big)^2}\esp\big[(\lambda_j'-\gamma_j^{m,n})^2\big]+O\big(n^{4}\varepsilon_n^{1/2}
\big(R_n^{(j)}\big)^{-2}
\big)+O(\varepsilon_n).\\
\end{equation*}
Then \eqref{eq_application_four_moment_thm_covariance} yields
\[\esp\big[(\lambda_j-\gamma_j^{m,n})^2\big]=\esp\big[(\lambda_j'-\gamma_j^{m,n})^2\big]+O\big(n^{2}\varepsilon_n^{1/2} + n^{-2}(R_n^{(j)})^2\varepsilon_n + (R_n^{(j)})^2n^{-c_0-2}\big). \]
As the first two error terms are smaller than the third one, the preceding equation becomes 
\begin{equation}\label{eq_difference_second_moment_covariance}
\esp\big[(\lambda_j-\gamma_j^{m,n})^2\big]=\esp\big[(\lambda_j'-\gamma_j^{m,n})^2\big]+O\big((R_n^{(j)})^2n^{
-c_0-2}\big).
\end{equation}

\subsection{Variance bounds}\label{bounds_Wishart}
\eqref{eq_difference_second_moment_covariance} is true for all eigenvalue $\lambda_j$. We estimate the error term $O\big((R_n^{(j)})^2n^{-c_0-2}\big)$ differently according to the location of the eigenvalue in the spectrum, in order to get the announced bounds.

\subsubsection{Inside the bulk of the spectrum}\label{bounds_Wishart_bulk}
Let $0< \eta \leqslant \frac{1}{2}$ and $\eta n\leqslant j \leqslant
(1-\eta)n$. From Theorem
\ref{thm_LUE_bulk},
$\esp\big[(\lambda_j'-\gamma_j^{m,n})^2\big] \leqslant C\frac{\log n}{n^2}$. Thus, from
\eqref{eq_difference_second_moment_covariance}, it remains to show that the error term is smaller than $\frac{\log n}{n^2}$.
But
\[R_n^{(j)}=(\log n)^{C\log\log n}n^{1/3}\min(j,n+1-j)^{-1/3} \leqslant \eta^{-1/3}(\log
n)^{C\log\log n}.\]
Then $(R_n^{(j)})^2n^{-c_0-2}=o_{\eta}\big(\frac{\log n}{n^2}\big)$. As a consequence, \[
\esp\big[(\lambda_j-\gamma_j^{m,n})^2\big]=\esp\big[(\lambda_j'-\gamma_j^{m,n})^2\big]+o_{\eta}\Big(\frac{\log
n}{n^2}\Big)
\]
and we get the desired result
\[
\esp\big[(\lambda_j-\gamma_j^{m,n})^2\big]\leqslant C\frac{\log n}{n^2},
\]
$C$ depending only on $\eta$, $A_1$ and $A_2$.

\subsubsection{Between the bulk and the edge of the spectrum}
Let $0< \eta \leqslant \frac{1}{2}$, $K > \kappa$ and $(1-\eta)n \leqslant j \leqslant n-K\log n$. From
Theorem
\ref{thm_LUE_intermediate},
$\esp\big[(\lambda_j'-\gamma_j^{m,n})^2\big] \leqslant C\frac{\log (n-j)}{n^{4/3}(n-j)^{2/3}}$. Thus, from
\eqref{eq_difference_second_moment_covariance}, it remains to show that the error term is smaller than $\frac{\log
(n-j)}{n^{4/3}(n-j)^{2/3}}$.
But
\[R_n^{(j)}=(\log n)^{C\log\log n}n^{1/3}(n+1-j)^{-1/3}.\]
Then $(R_n^{(j)})^2n^{-c_0-2}=o\big(\frac{\log (n-j)}{n^{4/3}(n-j)^{2/3}}\big)$. As a
consequence, \[
\esp\big[(\lambda_j-\gamma_j^{m,n})^2\big]=\esp\big[(\lambda_j'-\gamma_j^{m,n})^2\big]+o\Big(\frac{\log
(n-j)}{n^{4/3}(n-j)^{2/3}}\Big)
\]
and we get the desired result
\[
\esp\big[(\lambda_j-\gamma_j^{m,n})^2\big]\leqslant C\frac{\log
(n-j)}{n^{4/3}(n-j)^{2/3}},\]
$C$ depending only on $\eta$, $A_1$, $A_2$ and $K$. A similar result probably holds for the left-side of the spectrum, when $\rho>1$.

\subsubsection{At the edge of the spectrum}\label{bounds_Wishart_edge}
From Corollary \ref{thm_LUE_edge}, $\esp\big[(\lambda_n'-\gamma_n^{m,n})^2\big]=
\esp\big[(\lambda_n'-b_{m,n})^2\big]\leqslant Cn^{-4/3}$. By means
of \eqref{eq_difference_second_moment_covariance}, it remains to prove that the error term is smaller than $n^{-4/3}$.
We have \[R_n^{(n)}= (\log n)^{C\log\log n}n^{1/3}.\]
Consequently
$(R_n^{(n)})^2n^{-c_0-2}=o\big(n^{-4/3}\big)$. Then \[
\esp\big[(\lambda_n-b_{m,n})^2\big]=\esp\big[(\lambda_n'-2)^2\big]+o\big(n^{-4/3}\big)
\]
and
\[
\esp\big[(\lambda_n-2)^2\big]\leqslant Cn^{-4/3}.
\]
If this bound holds for the smallest eigenvalue $\lambda_1$ of LUE matrices, the same result is available for non Gaussian covariance matrices.

\section{Real Wishart matrices}\label{real_covariance}
The aim of this section is to prove Theorems \ref{thm_bulk_covariance}, \ref{thm_intermediate_covariance} and \ref{thm_edge_covariance} for real covariance matrices. The Four Moment Theorem (Theorem \ref{thm_4_moments_cov}) by Tao, Vu and Wang as well as Pillai and Yin's
Localization Theorem (Theorem \ref{thm_localization_covariance}) still hold for real covariance matrices. Section
\ref{section_covariance} is therefore valid for real matrices. The point is then to establish the results in the LOE case.

As announced in Section \ref{section_LUE_edge}, the variance of eigenvalues at the right edge of the spectrum is known to be bounded by $n^{-4/3}$ for LOE matrices (see \cite{LeRi_2010_deviations}). The conclusion for the largest eigenvalue is then established for large families of real covariance matrices.
\[\Var(\lambda_n) \leqslant \frac{\tilde{C}}{n^{4/3}}.\]

For eigenvalues in the bulk of the spectrum, following O'Rourke's approach (see \cite{Or_2010_fluctuations}), a Central Limit Theorem similar to the one established by Su in \cite{Su_2006_fluctuations} may be proved. In
particular, the normalization is still of the order of $(\frac{\log n}{n^2})^{1/2}$ and differs from
the complex case only by a constant. It is therefore natural to expect the same bound on the variance for LOE matrices.
The situation is completely similar for intermediate eigenvalues.
But LOE matrices do not have the same determinantal properties as LUE matrices, and it is therefore not clear
that a deviation inequality (similar to \eqref{eq_inequality_counting_function_lue}) holds for the eigenvalue
counting function. However, LOE and LUE matrices are
linked by interlacing formulas established by Forrester and Rains (see \cite{FoRa_2001_relationships}). These
formulas lead to the following relation between the eigenvalue counting functions in the complex and real cases: for all $t \in
\mathbb{R}$,

\begin{equation}\label{entrelacement_covariance}
\mathcal{N}_{t}(S_{m,n}^{\mathbb{C}})\overset{(d)}{=}\frac{1}{2}\big(\mathcal{N}_{t}(S_{m,n}^{\mathbb{R}})+\mathcal{N}_{t}(\tilde{S}
_{m,n}^{\mathbb{R}}
)\big)+\zeta_N(t) ,
\end{equation}
where
$S_{m,n}^{\mathbb{C}}$ is from the LUE,
$S_{m,n}^{\mathbb{R}}$, $\tilde{S}_{m,n}^{\mathbb{R}}$ are independent matrices from the LOE and $\zeta_N(t)$ takes values in
$\left\{-\frac{3}{2},-1,-\frac{1}{2},0,\frac{1}{2},1,\frac{3}{2}\right\}$.
See \cite{Or_2010_fluctuations} for more details.

The aim is now to establish a deviation inequality for the eigenvalue counting function similar to
\eqref{eq_inequality_counting_function_lue}. From \eqref{eq_inequality_counting_function_lue}, we know that for all $u
\geqslant 0$, 
\begin{equation*}
 \P\big(|\mathcal{N}_t(S_{m,n}^{\mathbb{C}})-n\mu_t|\geqslant u+C_1\big)\leqslant
2\exp\Big(-\frac{u^2}{2\sigma_t^2+u}\Big).
\end{equation*}
Set $C_1'=C_1+\frac{3}{2}$ and let $u \geqslant 0$. We can then write
\begin{align*}
\P\big(\mathcal{N}_t(S_{m,n}^{\mathbb{R}})- n & \mu_t \geqslant u+C_1'\big)^2 \\
& =  \P\Big(
\mathcal{N}_t(S_{m,n}^{\mathbb{R}})-n\mu_t \geqslant u+C_1', \ \mathcal{N}_t(\tilde{S}_{m,n}^{\mathbb{R}})-n\mu_t
\geqslant u+C_1'\Big)\\
 & \leqslant 
\P\Big(\tfrac{1}{2}\big(\mathcal{N}_t(S_{m,n}^{\mathbb{R}})+\mathcal{N}_t(\tilde{S}_{m,n}^{\mathbb{R}})\big)-n\mu_t
\geqslant u+C_1'\Big)\\
 & \leqslant  \P\big(\mathcal{N}_t(S_{m,n}^{\mathbb{C}})-n\mu_t \geqslant u+C_1'-\frac{3}{2}\big)\\
 & \leqslant  2\exp\Big(-\frac{u^2}{2\sigma_t^2+u}\Big).
\end{align*}
Repeating the computations for $\P\big(\mathcal{N}_t(S_{m,n}^{\mathbb{R}})-n\mu_t \leqslant -u-C_1'\big)$ and
combining with the preceding yield
\begin{equation}\label{real_inequality_counting_function_covariance}
 \P\big(|\mathcal{N}_t(S_{m,n}^{\mathbb{R}})-n\mu_t| \geqslant u+C_1'\big) \leqslant
2\sqrt{2}\exp\Big(-\frac{u^2}{4\sigma_t^2+2u}\Big).
\end{equation}
Note that $\sigma_t^2$ is still the variance of $\mathcal{N}_t(S_{m,n}^{\mathbb{C}})$ in the preceding formula.

What remains then to be proved is very similar to the complex case. From
\eqref{real_inequality_counting_function_covariance} and Su's bounds on the variance $\sigma_t^2$ (see
\eqref{sup_variance_lue} for the bulk case and \eqref{sup_variance_lue_interm} for the intermediate
case), deviation inequalities for individual eigenvalues can be deduced, as was done to prove
Propositions~\ref{prop_deviation_lue_bulk} and \ref{prop_deviation_lue_interm}.
It is
then straightforward to derive the announced bounds on the variances for LOE matrices.
The argument developed in Section \ref{section_covariance} in order to extend the LUE results to large families of
covariance matrices can be reproduced to reach the desired bounds on the variances of eigenvalues in
the bulk and between the bulk and the edge of the spectrum for families of real covariance matrices.

\section[Rate of convergence]{Rate of convergence towards the Marchenko-Pastur distribution}\label{section_wasserstein_covariance}
In this whole section, we suppose that Theorem \ref{thm_intermediate_covariance} holds for left-side intermediate eigenvalues.
The bounds on $\esp[(\lambda_j-\gamma_j^{m,n})^2]$ developed in the preceding sections lead to a bound on the rate of
convergence of the empirical spectral measure $L_{m,n}$ towards the Marchenko-Pastur distribution in terms of
$2$-Wasserstein distance. Recall that $W_2(L_{m,n},\mu_{m,n})$ is a random variable defined by
\[W_2(L_{m,n},\mu_{m,n})
    =\inf \bigg (\int_{\mathbb{R}^2}|x-y|^2\,d\pi(x,y)\bigg)^{1/2},\]
where the infimum is taken over all probability measures $\pi$ on $\mathbb{R}^2$ with respective
marginals $L_{m,n}$ and $\mu_{m,n}$. To achieve the expected bound, we rely
on another expression of $W_2$ in terms of distribution functions, namely
\begin{equation}\label{w2_fct_repartition_covariance}
 W_2^2(L_{m,n},\mu_{m,n}) = \int_0^1\big(F_{m,n}^{-1}(x)-G_{m,n}^{-1}(x)\big)^2\,dx,
\end{equation}
where $F_{m,n}^{-1}$ (respectively $G_{m,n}^{-1}$) is the generalized inverse of the distribution function $F_{m,n}$
(respectively $G_{m,n}$) of $L_{m,n}$ (respectively $\mu_{m,n}$) (see for example \cite{Vi_2003_book}). 
On the basis of this representation, the following statement may be derived.

\begin{prop}\label{prop_w2_covariance}
There exists a constant $C>0$ depending only on $A_2$ such that for all $1 \leqslant \frac{m}{n} \leqslant A_2$,
\begin{equation}\label{borne_ps_w2_covariance}
 W_2^2(L_{m,n},\mu_{m,n}) \leqslant \frac{2}{n}\sum_{j=1}^n
(\lambda_j-\gamma_j^{m,n})^2+\frac{C}{n^2}.
\end{equation}
\end{prop}
\begin{proof} From \eqref{w2_fct_repartition_covariance},
\[W_2^2(L_{m,n},\mu_{m,n}) = \int_0^1\big(F_{m,n}^{-1}(x)-G_{m,n}^{-1}(x)\big)^2\,dx.\] Then,
\begin{align*}
 W_2^2(L_{m,n},\mu_{m,n}) & = \sum_{j=1}^n\int_{\frac{j-1}{n}}^{\frac{j}{n}}\big(\lambda_j-G_{m,n}^{-1}(x)\big)^2\,dx\\
 & \leqslant \frac{2}{n}\sum_{j=1}^n(\lambda_j-\gamma_j^{m,n})^2 +
2\sum_{j=1}^n\int_{\frac{j-1}{n}}^{\frac{j}{n}}\big(\gamma_j^{m,n}-G_{m,n}^{-1}(x)\big)^2\,dx.
\end{align*}
But $\gamma_j^{m,n}=G_{m,n}^{-1}\big(\frac{j}{n}\big)$ and $G_{m,n}^{-1}$ is non-decreasing. Therefore,
$\big|\gamma_j^{m,n}-G_{m,n}^{-1}(x)\big| \leqslant
\gamma_{j}^{m,n}-\gamma_{j-1}^{m,n}$ for all $x \in \big[\frac{j-1}{n},\frac{j}{n}\big]$. Consequently,
\begin{equation}\label{inegalite_w2_intermediaire_covariance}
W_2^2(L_{m,n},\mu_{m,n}) \leqslant \frac{2}{n}\sum_{j=1}^n(\lambda_j-\gamma_j^{m,n})^2 +
\frac{2}{n}\sum_{j=1}^n(\gamma_j^{m,n}-\gamma_{j-1}^{m,n})^2.
\end{equation}
But if $j-1 \geqslant \frac{n}{2}$,
\begin{align*}
\frac{1}{n} & = \int_{\gamma_{j-1}^{m,n}}^{\gamma_j^{m,n}}\frac{1}{2\pi x}\sqrt{(b_{m,n}-x)(x-a_{m,n})}\,dx\\
  & \geqslant \frac{\sqrt{\gamma_{j-1}^{m,n}-a_{m,n}}}{2\pi b_{m,n}}\int_{\gamma_{j-1}^{m,n}}^{\gamma_j^{m,n}}\sqrt{b_{m,n}-x}\,dx\\
 & \geqslant
\frac{\sqrt{\gamma_{j-1}^{m,n}-a_{m,n}}}{3\pi b_{m,n}}(b_{m,n}-\gamma_{j-1}^{m,n})^{3/2}\bigg(1-\Big(1-\frac{\gamma_j^{m,n}-\gamma_{j-1}^{m,n}}{b_{m,n}-\gamma_{j-1}^{m,n}}\Big)^{
3/2 }\bigg)\\
  & \geqslant
 \frac{1}{3b_{m,n}\sqrt{b_{m,n}-a_{m,n}}}(b_{m,n}-\gamma_{j-1}^{m,n})^{1/2}(\gamma_j^{m,n}-\gamma_{j-1}^{m,n})\\
 & \geqslant
C(A_2)\Big(\frac{n-j+1}{n}\Big)^{1/3}(\gamma_j^{m,n}-\gamma_{j-1}^{m,n}),
\end{align*}
from \eqref{eq_distance_gamma_bord_droit_dependant_j}. Then
\[\gamma_j^{m,n}-\gamma_{j-1}^{m,n} \leqslant \frac{C(A_2)}{n^{2/3}(n-j+1)^{2/3}}.\]
It may be shown that a similar bound holds if $j-1 \leqslant \frac{n}{2}$. As a summary, there exists a
constant $c>0$ depending only on $A_2$ such that, for all $j \geqslant
2$,
\begin{equation}
\gamma_j^{m,n}-\gamma_{j-1}^{m,n} \leqslant \frac{c}{n^{2/3}\min(j,n+1-j)^{1/3}}.
\end{equation}
This yields
\begin{equation*}
 \sum_{j=1}^n(\gamma_j^{m,n}-\gamma_{j-1}^{m,n})^2  \leqslant 
\frac{c^2}{n^{4/3}}\sum_{j=1}^n\frac{1}{\min(j,n+1-j)^{2/3}}\leqslant \frac{C}{n}, 
\end{equation*}
where $C>0$ depends only on $A_2$.
Then \eqref{inegalite_w2_intermediaire_covariance} becomes
\[W_2^2(L_{m,n},\mu_{m,n}) \leqslant \frac{2}{n}\sum_{j=1}^n
(\lambda_j-\gamma_j^{m,n})^2+\frac{C}{n^2},\]
where $C>0$ depends only on $A_2$, which is the claim.
\end{proof}

\begin{proof}[Proof of Corollary \ref{cor_wasserstein_2_covariance}]
Let $n \geqslant 2$. Due to Proposition \ref{prop_w2_covariance},
\[\esp\big[W_2^2(L_{m,n},\mu_{m,n})\big] \leqslant \frac{2}{n}\sum_{j=1}^n
\esp\big[(\lambda_j-\gamma_j^{m,n})^2\big]+\frac{C}{n^2}.\]
We then make use of the bounds on $\esp\big[(\lambda_j-\gamma_j^{m,n})^2\big]$
produced in the previous sections.
Set $0< \eta \leqslant \frac{1}{2}$ and $K>\kappa$ so that $K\log n \leqslant \eta n$.
We first decompose
\begin{align*}
\sum_{j=1}^n \esp\big[(\lambda_j-\gamma_j^{m,n})^2\big] & = \sum_{j=1}^{K\log
n}\esp\big[(\lambda_j-\gamma_j^{m,n})^2\big] + \sum_{j=K\log n+1}^{\eta
n}\esp\big[(\lambda_j-\gamma_j^{m,n})^2\big]\\
 & + \sum_{j=\eta n+1}^{(1-\eta)n-1}\esp\big[(\lambda_j-\gamma_j^{m,n})^2\big] + \sum_{j=(1-\eta)n}^{n-K\log
n-1}\esp\big[(\lambda_j-\gamma_j^{m,n})^2\big]\\
 & +\sum_{j=n-K\log n}^{n}\esp\big[(\lambda_j-\gamma_j^{m,n})^2\big]\\
 & = \Sigma_1 + \Sigma_2 + \Sigma_3 + \Sigma_4 + \Sigma_5.
\end{align*}
The sum $\Sigma_3$ will be bounded using the bulk case (Theorem \ref{thm_bulk_covariance}), while Theorem \ref{thm_intermediate_covariance}
will be
used to handle $\Sigma_2$ and $\Sigma_4$. 
A crude version of Theorem \ref{thm_localization_covariance} will be enough to
bound $\Sigma_1$ and $\Sigma_5$. To start with thus, from Theorem \ref{thm_bulk_covariance},
\[\Sigma_3 \leqslant \sum_{j=\eta n+1}^{(1-\eta)n-1}C\frac{\log n}{n^2} \leqslant C\frac{\log
n}{n}.\]
Secondly, from Theorem \ref{thm_intermediate_covariance},
\begin{equation*}
\Sigma_2+\Sigma_4  \ \leqslant \ \frac{C}{n^{4/3}}\sum_{j=K\log n+1}^{\eta
n}\frac{\log j}{j^{2/3}} \ \leqslant \ C\frac{\log n}{n}.
\end{equation*}
Next $\Sigma_1$ and $\Sigma_5$ have only $K\log n$ terms. If each term is bounded by $\frac{C}{n}$ where
$C$ is a positive universal constant, we get that $\Sigma_1+\Sigma_5 \leqslant \frac{2KC\log n}{n}$,
which is enough to prove the desired result on $\sum_{j=1}^n\esp\big[(\lambda_j-\gamma_j^{m,n})^2\big]$. For $n$
large
enough depending only on constant $C$ in Theorem \ref{thm_localization_covariance}, $\frac{1}{\sqrt{n}} \geqslant
\frac{(\log n)^{C\log\log n}}{n^{2/3}\min(j,n+1-j)^{1/3}}$ and Theorem \ref{thm_localization_covariance} yields
\[\P\Big(|\lambda_j-\gamma_j^{m,n}| \geqslant \frac{1}{\sqrt{n}}\Big) \leqslant
Ce^{-(\log n)^{c\log \log n}}.\] Then, by the Cauchy-Schwarz inequality,
\begin{align*}
 \esp\big[(\lambda_j-\gamma_j^{m,n})^2\big] & \leqslant
\esp\Big[(\lambda_j-\gamma_j^{m,n})^2\mathbbm{1}_{|\lambda_j-\gamma_j^{m,n}|\leqslant \frac{1}{\sqrt{n}}}\Big]+
\esp\Big[(\lambda_j-\gamma_j^{m,n})^2\mathbbm{1}_{|\lambda_j-\gamma_j^{m,n}|> \frac{1}{\sqrt{n}}}\Big]\\
 & \leqslant
\frac{1}{n}+\sqrt{\esp\big[|\lambda_j-\gamma_j^{m,n}|^4\big]}\sqrt{\P\Big(|\lambda_j-\gamma_j^{m,n}| >
\frac{1}{\sqrt{n}}\Big)}\\
& \leqslant
\frac{1}{n}+\sqrt{3}C n^{2}e^{-(\log n)^{c\log \log n}}.
\end{align*}
As $\sqrt{3}C n^{2}e^{-(\log n)^{c\log \log n}}=o(\frac{1}{n})$, there exists a constant $C>0$ such that
$\esp\big[(\lambda_j-\gamma_j^{m,n})^2\big] \leqslant \frac{C}{n}$.
Then
\[\Sigma_1+\Sigma_5 \leqslant 2KC\frac{\log n}{n}.\]
As a consequence,
\[\sum_{j=1}^n \esp\big[(\lambda_j-\gamma_j^{m,n})^2\big] \leqslant C\frac{\log n}{n}.\]
Therefore
\[\esp\big[W_2(L_{m,n},\mu_{m,n})^2\big] \leqslant C\frac{\log n}{n^2},\]
where $C>0$ depends only on $A_1$ and $A_2$, which is the claim. The corollary is thus established.
\end{proof}

\subsection*{Acknowledgements}
I would like to thank my advisor, Michel Ledoux, for bringing up this problem to me, for the several discussions we had about this work, and Emmanuel Boissard for very useful conversations about Wasserstein distances.

\bibliographystyle{plain}


\end{document}